\documentclass[sn-mathphys,Numbered,crop]{sn-jnl}

\usepackage{multirow}%
\usepackage{amsmath,amssymb,amsfonts}%
\usepackage{amsthm}%
\usepackage{mathrsfs}%
\usepackage[title]{appendix}%
\usepackage[dvipsnames]{xcolor}
\usepackage{textcomp}%
\usepackage{manyfoot}%
\usepackage{booktabs}%
\usepackage{algorithm}%
\usepackage{algorithmicx}%
\usepackage{algpseudocode}%
\usepackage{listings}%
\usepackage{amsmath}

\theoremstyle{thmstyleone}%
\theoremstyle{thmstyletwo}%

\theoremstyle{thmstylethree}%

\usepackage{amsmath} 
\usepackage{amssymb}  

\usepackage{color}
\usepackage{pgfplots}
\usepackage{tikz}
\usepackage{tikzscale}
\usepackage{placeins}
\usepackage{subcaption}

\usepackage{array}
\usepackage{booktabs}
\usepackage{multirow}
\usepackage{csquotes}
\usepackage{dblfloatfix}
\usepackage{float}
\usepackage{nicefrac}

\usepackage{enumitem}

\usepackage{orcidlink}

\newcommand{\vek}[1]{{\mathbf #1}}

\newcommand{\sv}[1]{\boldsymbol{#1}}

\newcommand{\play}[1]{^{(#1)}}
\newcommand{\Tplay}[1]{^{(#1)^\mathsf{T}}}
\newcommand{\Invplay}[1]{^{(#1)-1}}
\newcommand{\Splay}[1]{^{*(#1)}}
\newcommand{\eg}{e.g.~}

\newtheorem{deff}{Definition}
\newtheorem{collar}{Remark}
\newtheorem{prob_stat}{Problem}
\newtheorem{lem}{Lemma}

\usepackage{fancyhdr}
\fancypagestyle{firstpage}{
	\fancyhf{}

	\fancyfoot[C]{\hspace*{3mm}\small{This work was supported by the Federal Ministry for Economic Affairs and Climate Action, in the New Vehicle and System Technologies research initiative with Project number 19A21008D.}\hspace*{3mm}}
}

\begin{document}
	
	\pagenumbering{gobble} 
	
	\title[Identification Methods for Ordinal Potential Differential Games]{Identification Methods for Ordinal Potential Differential Games}

	\author*{\sur{Balint} \fnm{Varga}$^*$ \orcidlink{0000-0002-4189-0105}} \email{balint.varga2@kit.edu}

	\author{\sur{Da} \fnm{Huang}}
	
	\author{\sur{S\"oren} \fnm{Hohmann}}
	
	\affil{\orgdiv{ Institute of Control Systems (IRS)}, \orgname{Karlsruhe Institute of Technology}, \orgaddress{\city{Karlsruhe}, \postcode{76131}, \country{Germany}}}
	

	\abstract{
		This paper introduces two new identification methods for linear quadratic (LQ) ordinal potential differential games (OPDGs). Potential games are notable for their benefits, such as the computability and guaranteed existence of Nash Equilibria. While previous research has analyzed ordinal potential static games, their applicability to various engineering applications remains limited. Despite the earlier introduction of OPDGs, a systematic method for identifying a potential game for a given LQ differential game has not yet been developed. To address this gap, we propose two identification methods to provide the quadratic potential cost function for a given LQ differential game. Both methods are based on linear matrix inequalities (LMIs). The first method aims to minimize the condition number of the potential cost function's parameters, offering a faster and more precise technique compared to earlier solutions. In addition, we present an evaluation of the feasibility of the structural requirements of the system. The second method, with a less rigid formulation, can identify LQ OPDGs in cases where the first method fails. These novel identification methods are verified through simulations, demonstrating their advantages and potential in designing and analyzing cooperative control systems.
	}
	


	\maketitle

	\newpage
	\section{Introduction}
	Game theory is commonly used to model the interactions between rational agents~\cite{1998_DynamicNoncooperativeGame_basar,2005_LQDynamicOptimization_engwerda,2016_PotentialGameTheory_la}, which arises in a wide range of applications, such as modeling the behavior of companies in the stock market \cite{2001_GameTheoryBusiness_chatterjee}, solving routing problems in communication networks \cite{2006_AdaptiveChannelAllocation_nie}, and studying human-machine cooperation \cite{2014_NecessarySufficientConditions_flad}.
	Utilizing game theory's systematic modeling and design techniques, effective solutions for various applications have been analyzed and designed, see e.g.~\cite{2021_OnlineInverseLinearQuadratic_inga}. These applications often hinge on the strategic decisions of participants, making the equilibrium concepts crucial for predicting and understanding the outcomes of such interactions.
	
	One of the core equilibrium concepts of game theory is the so-called Nash Equilibrium (NE), cf.~\cite{1998_DynamicNoncooperativeGame_basar}. In an NE of a game, each player reaches their optimal utility, taking the optimal actions of all other players into account. If the players find themselves in a Nash Equilibrium, there is no rational reason for them to change their actions.
	
	The very first idea of a fictitious function replacing the original structure of a non-cooperative strategic game with $N$ players was given by \textit{Rosenthal} \cite{1973_ClassGamesPossessing_rosenthal}.
	Based on this idea, the formal definitions of potential games were first introduced by \textit{Moderer and Shapley} \cite{1996_PotentialGames_monderer}. The general concept of \textit{potential games} is presented visually in Figure~\ref{fig:PG_general_idae}: The original non-cooperative strategic game\footnote{Note it is assumed that the original game is given, therefore the term \textit{given game} is used interchangeably to emphasize this property of the original game.} with $N$ players and with their cost functions {$J\play{i},$~$\,$~$i=1...N$} are replaced with one single potential function. This potential function provides a single mapping of strategy space $\, \mathcal{U} =  \mathcal{U}\play{1} \times ... \times \mathcal{U}\play{N}$ of the original game to the real numbers 
	\begin{equation} \label{eq:chap4_pot_intro}
		J\play{p} : \mathcal{U} \rightarrow \mathbb{R},
	\end{equation}
	instead of $N$ mappings of the combined strategy set of the players to the set of real numbers
	\begin{equation} \label{eq:chap4_J_all_i}
		J\play{i} : \mathcal{U} \rightarrow \mathbb{R}, \forall i \in \mathcal{P},
	\end{equation}
	where $ \mathcal{U}\play{i}$ represents the strategy space of player $i$. Therefore, the potential function serves as a substitute model for the original game, while retaining all essential information of the original game. Intuitively, the Nash equilibrium (NE) of the non-cooperative strategic game can be more easily computed using the potential function~\eqref{eq:chap4_pot_intro} than the coupled optimization of \eqref{eq:chap4_J_all_i} of the original game. A further beneficial feature of potential games is that they possess at least one NE. Furthermore, if the potential function is strictly concave and bounded, the NE is unique. These advantageous properties make potential games a highly appealing tool for strategic game analysis, which leads to advantages in engineering application with multiple decision makers, cf.~\cite{2016_SurveyStaticDynamic_gonzalez-sanchez}, \cite[Section 2.2]{2016_PotentialGameTheory_la}. These properties and application domains make further research on potential games interesting for the research community.
	
	Potential games have been extensively analyzed in the literature~\cite{1997_CharacterizationOrdinalPotential_voorneveld, 1999_PotentialGamesPurely_kukushkin,2012_StateBasedPotential_marden}. However, the focus of these works is primarily on \textit{potential static games}, see \eg \cite{2018_PotentialGamesDesign_Cheng}, which do not incorporate underlying dynamical systems.
	In contrast, \textit{differential games} are particularly valuable for modeling and controlling cooperative systems in engineering applications, as they account for dynamic systems, see e.g.~\cite{2021_OnlineInverseLinearQuadratic_inga}. However, the existing literature on \textit{potential differential games} exclusively deals with \textit{exact potential differential games}, which have such a definition which limits their general applicability.
	
	As a result, exact potential differential games can only describe games with specific system structures or utility functions for the players. For instance, the cost functions of the players are symmetric \cite{2016_SurveyStaticDynamic_gonzalez-sanchez, 2018_PotentialDifferentialGames_fonseca-morales}. Consequently, a general usage of exact potential differential games can be limited.
	
	To allow a broader usage, the subclass of linear quadratic (LQ) \textit{ordinal potential differential games} (OPDGs) has been previously introduced in literature \cite{2021_PotentialDifferentialGames_varga}. An OPDG can be used to model various engineering problems, like resources management of power networks \cite{8329563, 2024_GametheoreticalParadigmCollaborative_ta} and multi-agent interactions \cite{10342328, 2024_IntAweDec_varga}. 
	The class of ordinal potential static games has been explored in the literature, including verification and identification methods for finite potential static games in \cite{2014_onFinPot_Cheng, 2018_PotentialGamesDesign_li}. However, a systematic identification method for OPDGs is still absent from the literature. The identification\footnote{In this paper, the term \textit{identification} refers to determining whether a given differential game can be classified as a potential differential game.} is a crucial step in the analysis and design of potential games, as it enables us to derive the potential function for a given game structure. The absence of a systematic identification method for OPDGs limits their applicability in engineering and other fields. Therefore, there is a need for a new identification method that can provide a potential function for a given differential game structure and extend the reach of OPDGs to various applications.

	\begin{figure}[t!]
		\centering
		\includegraphics[width=0.98\linewidth]{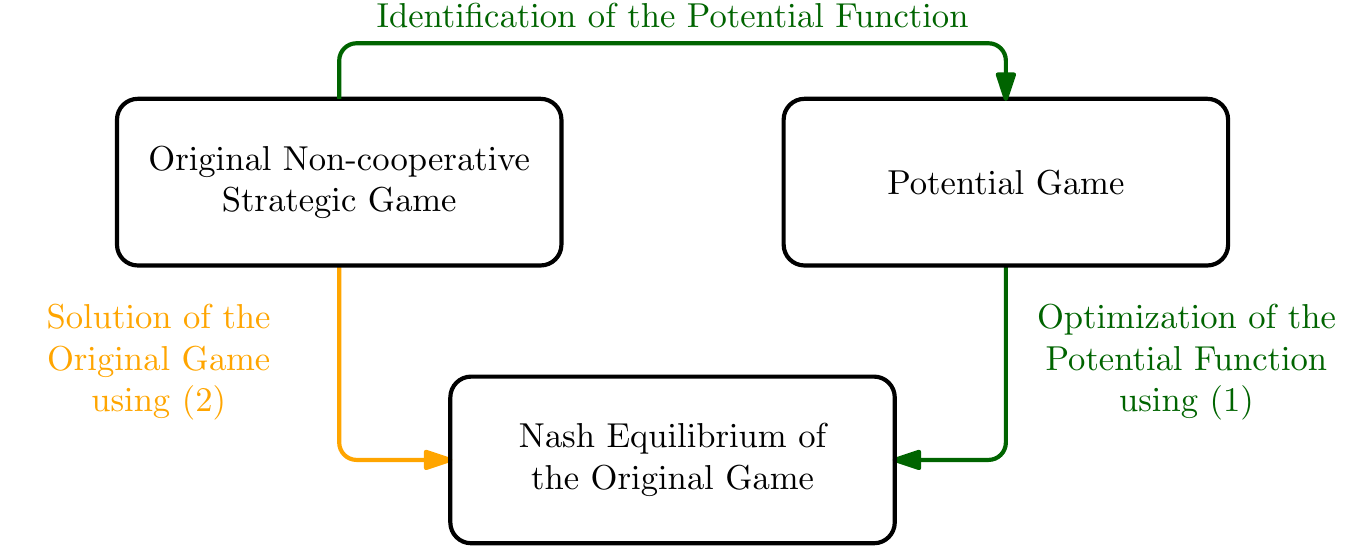}
		\caption{The illustration of the general idea of the potential games, where the original game is replaced by a (fictitious) potential function. The optimum of this potential function provides the NE of the original game.}
		\label{fig:PG_general_idae}
	\end{figure}
	
	Therefore, the contribution of this paper is the development of two identification methods to find the potential function of an OPDG: A potential function is identified for a given LQ differential game using linear matrix inequalities (LMI). Using an LMI provides a fast and accurate solution. The first identification method requires only the cost functions of the players and no information about the system state trajectories of the differential game is needed to determine the potential function. However, in some cases, the solution is not feasible since the constraints of the LMI are violated. Therefore, a further identification method is proposed, in which the constraints are softened. In order to construct the potential function using the second identification method, the system state trajectories of the differential game must be available. This has the disadvantage that the identification method is less robust against measurement noise of the trajectories compared to the first identification method. 
	
	The paper is structured as follows: Section II provides an overview of LQ differential games. Then, the identification methods, are given in Section III. Section~IV presents the application of the proposed identification methods to two examples using simulations. Finally, Section V offers a brief summary and outlook.

	\section{Preliminaries - Differential Games}
	First, this section provides a short overview of linear-quadratic (LQ) differential games. Then, the core idea of potential games and the class of the exact potential differential games are presented. Finally, the core notion of OPDGs is provided and the limitations of the state of the art are discussed
	\subsection{Linear-Quadratic Differential Games}
	
	The general idea of a game is that numerous players interact with each other and try to optimize their cost function\footnote{In literature of game theory, the formulation as a minimization problem is usual \cite[Chapter 1-3.]{2007_AlgorithmicGameTheory_nisan}. Since optimal control theory, and minimization problems are prevalent \cite{2015_OptimierungStatischeDynamische_papageorgiu}, the optimization problems in this paper are formulated as minimization.}. If the players also interact with a dynamics system, the game is called \textit{differential game}, see \cite{1998_DynamicNoncooperativeGame_basar, 2016_NonzeroSumDifferentialGames_basar}. In this case, the players have to carry out dynamic optimizations.
	
	Numerous practical engineering applications can often be characterized sufficiently with such LQ models, e.g.~\cite{2021_FeedbackSystemsIntroduction_astrom, 2021_InverseDynamicGame_ingacharaja, 2022_FeedbackControlSystems_rahmani-andebili}. An LQ differential game ($\Gamma_\text{LQ}$) is defined as a tuple~of:
	\begin{itemize}
		\item a set $\mathcal{P}$ of players $i=1...N$ with their cost functions $J\play{i}$,
		\item inputs of the players $\sv{u}(t) = \left[\sv{u}(t)\play{1}, ...,\sv{u}(t)\play{N}\right],$
		\item a dynamic system $\sv{f}(t,\sv{x}(t), \sv{u}(t)) = \dot{\sv{x}}(t)$,
		\item the system states $\sv{x}(t) \in \mathbb{R}$.
	\end{itemize}		
	The linear dynamic system is
	\begin{align} \label{eq:linear_system}
		\dot{\sv{x}} &= \mathbf{A}\sv{x}+ \sum_{i=1}^{N} \vek{B}\play{i}\sv{u}\play{i}, \\ \nonumber
		\sv{x}&(t_0) = \sv{x}_0,
	\end{align}
	where $\mathbf{A}$ and $\mathbf{B}^{(i)}$ are the system matrix and the input matrices of each player, respectively\footnote{Note that for the sake of simplicity, the time dependency of $\sv{x}(t)$ and $\sv{u}(t)$ are omitted in the following.}. Furthermore, the cost function of each player $i$ is given in a quadratic form:
	\begin{equation} \label{eq:each_player_quad_cost_function}
		J\play{i}=\frac{1}{2}\int_{t_0}^{\infty} \sv{x}^\mathsf{T} \vek{Q}\play{i}\sv{x} + \sum_{j=1}^{N} {\sv{u}\play{j}}^\mathsf{T} \vek{R}\play{ij}\sv{u}\play{j} \text{d}t,
	\end{equation}
	where matrices $\mathbf{Q}^{(i)}$ and $\mathbf{R}^{(ij)}$ denote the penalty factor on the state variables and on the inputs, respectively. The matrices $\mathbf{Q}^{(i)}$ are positive semi-definite and $\mathbf{R}^{(ij)}$ are positive definite.
	A common solution concept of games is the NE, see e.g.~\cite{2005_LQDynamicOptimization_engwerda}, which is defined in the following:
	\begin{deff} [Nash Equilibrium]
		The solution of 
		\begin{align} 
			\sv{u}\Splay{i} =  \underset{\sv{u}\play{i}}{\mathrm{arg min}} \;J\play{i}\left({\vek{u}\play{1}}^{*}, ...,\vek{u}\play{i}, ..., {\vek{u}\play{N}}^{*}  \right), \; \; \;  \forall i \in \mathcal{P}, \; \; \;  \; \rm{s.t.} \;  \text{(\ref{eq:linear_system})}
		\end{align}
		is called the Nash Equilibrium of the differential game composed by \eqref{eq:linear_system} and \eqref{eq:each_player_quad_cost_function}. 
	\end{deff}
	The NE is a stable solution of the game, i.e. if the players deviate from this equilibrium solution, they face higher costs.
	In LQ games, the NE can be computed by the coupled algebraic Riccati equations of the differential game. Therefore, the NE of the game can be characterized by the resulting players' inputs $\sv{u}\play{i}$. In case of a feedback structure, the control law $\sv{u}\play{i} = - \sv{K}\play{i}\sv{x}$ can be assumed. This feedback gain is computed by the algebraic Riccati equation (see e.g.~\cite{2005_LQDynamicOptimization_engwerda}.)
	\begin{align} \label{eq:coup_Ric}
		\vek{0} =  &\vek{A}^\mathsf{T} \vek{P}\play{i} +  \vek{P}\play{i} \vek{A} + \vek{Q}\play{i} -
		\sum_{j \in \mathcal{P}} \vek{P}\play{i}\vek{S}\play{j}\vek{P}\play{j} \\ \nonumber
		&-\sum_{j \in \mathcal{P}} \vek{P}\play{j}\vek{S}\play{j}\vek{P}\play{i} + 
		\sum_{j \in \mathcal{P}} \vek{P}\play{j}\vek{S}\play{ij}\vek{P}\play{j} , \; \forall i \in \mathcal{P},
	\end{align}
	where the following simplifications are applied
	\begin{align*}
		\vek{S}\play{j}  =& \vek{B}\play{j} {\vek{R}\play{jj}}^{-1} {\vek{B}\play{j}}^{T}	\; j \in \mathcal{P} \\
		\vek{S}\play{ij} =& \vek{B}\play{j} {\vek{R}\play{jj}}^{-1} \vek{R}\play{ij} {\vek{R}\play{jj}}^{-1} {\vek{B}\play{j}}^{T}\; j \in \mathcal{P}, i \neq j.
	\end{align*}
	From the solution $\vek{P}\play{i}$, the feedback gain of the players are computed
	\begin{equation}  \label{eq:KfromRicc_res}
		\vek{K}\play{i} = {\vek{R}\play{ii}}^{-1} \vek{B}\Tplay{i} \vek{P}\play{i}.
	\end{equation}
	Note that for LQ Differential Games, the solution $\vek{P}\play{i}$ and consequently the feedback gain $\vek{K}\play{i}$ are unique.  
	\subsection{Potential Differential Games}
	As mentioned in the introduction, potential games have a very useful property: The computation of the NE can be reduced to a single optimization problem of one cost (fictitious) function $J\play{p}$. Such a potential game can be considered as a substituting optimal controller of the associated LQ differential game $\Gamma_\text{LQ}$.
	The potential function is assumed to be quadratic
	\begin{equation} \label{eq:odpg_obj}
			J^{(p)}=\frac{1}{2} \int_{t_0}^{\infty} \sv{x}^{T} \mathbf{Q}^{(p)} \sv{x}+\sv{u}^{T} \mathbf{R}^{(p)} \sv{u} \, \mathrm{d} t
	\end{equation}
	where the matrices $\mathbf{Q}^{(p)}$ are positive semi-definite and $\mathbf{R}^{(p)}$ are positive definite, respectively. The vector $\sv{u}= \left[\sv{u}\play{1}, ...,\sv{u}\play{N}\right]$ involves all the players' inputs.	In an LQ case, the Hamiltonian of the potential function is
	\begin{equation} \label{eq:hamiltionian_of_pot_game}
		H\play{p} = \frac{1}{2} \sv{x}^{T} \mathbf{Q}^{(p)} \sv{x}+\frac{1}{2} \sv{u}^{T} \mathbf{R}^{(p)} \sv{u}+\sv{\lambda}^{(p)^{T}} \dot{\sv{x}} 
	\end{equation}
	and the Hamiltonian of the player $i$ of $\Gamma_\text{LQ}$ is
	\begin{align} \label{eq:hamiltionian_of_player_i}
		H\play{i} = \frac{1}{2}  \sv{x}^\mathsf{T} \vek{Q}\play{i}\sv{x} + \frac{1}{2} \displaystyle  \sum_{j=1}^{N} {\sv{u}\play{j}}^\mathsf{T} \vek{R}\play{ij}\sv{u}\play{j} 	+\sv{\lambda}^{(i)^{T}} \! \! \dot{\sv{x}} \; \; \forall i \in \mathcal{P}.
	\end{align}
	
	\begin{deff} [Exact Potential Differential Game \cite{1996_PotentialGames_monderer}] \label{def:EPDG}
		An exact potential differential game is a game, in which 
		\begin{align} \label{eq:def_pot_game}
			\frac{\partial H^{(p)}}{\partial \boldsymbol{u}^{(i)}}=\frac{\partial H^{(i)}}{\partial \boldsymbol{u}^{(i)}} \; \; \forall i \in \mathcal{P},
		\end{align}
		holds, where $H\play{p}$ and $H\play{i}$ are given in (\ref{eq:hamiltionian_of_pot_game}) and in (\ref{eq:hamiltionian_of_player_i}), respectively. 
	\end{deff}

A valuable property of an exact potential differential game that if (\ref{eq:def_pot_game}) holds, then optimization
\begin{align} \label{eq:EPDG_optim}
	\sv{u} =&  \underset{\sv{u}}{\mathrm{arg min}} \;J\play{p}, \\ \nonumber
	& \rm{s.t.} \; \text{(\ref{eq:linear_system})},
\end{align}
yields the NE of $\Gamma_{\text{LQ}}$. Thus, the computation of the NE can happen by the (\ref{eq:odpg_obj}) instead of \eqref{eq:coup_Ric}.

	In \cite{2016_SurveyStaticDynamic_gonzalez-sanchez} and \cite{2018_PotentialDifferentialGames_fonseca-morales}, exact potential differential games are analyzed, and their identification methods are provided on how to find the exact potential function of a given differential game. Using the optimization \eqref{eq:EPDG_optim}, the stabilizing feedback control law 
	\begin{equation}
		\sv{u} = \vek{K}\play{p}\sv{x}
	\end{equation}
	of the potential games is obtained. The feedback gain is
	\begin{equation}
		\vek{K}\play{p} = {\vek{R}\play{p}}^{-1}\vek{B}\play{p} \vek{P}\play{p},
	\end{equation}
	where $\vek{B}\play{p} = \left[\vek{B}\play{1}, \vek{B}\play{2}, ..., \vek{B}\play{N} \right]$. $\vek{P}\play{p}$ is the solution of the Riccati equation obtained from the optimization of \eqref{eq:odpg_obj}.
	\vspace*{5mm}	
	\subsection{Ordinal Potential Differential Games}
	The main drawback of exact potential differential games is their limited applicability for general engineering problems, see e.g.~\cite{2016_DynamicPotentialGames_zazo}. Due to its definition, exact potential differential games can be solely applied to games 
	\begin{itemize}
		\item with special system dynamics (e.g.~$\sv{u}\play{i}$ has only impact on $\sv{x}_i$, see Assumptions Theorem~1 and Corollary~1 from \cite{2018_PotentialDifferentialGames_fonseca-morales}) or
		\item for which the cost functions of the players have a particular structure: E.g.~elements in main diagonal are identical for all players, see~\cite{2016_SurveyStaticDynamic_gonzalez-sanchez} or Example 4 in~\cite{2018_PotentialDifferentialGames_fonseca-morales}.
	\end{itemize}
	To allow a broader usage, the subclass of OPDGs has been introduced in~\cite{2021_PotentialDifferentialGames_varga}. 
	\begin{deff}[Ordinal Potential Differential Game \cite{2021_PotentialDifferentialGames_varga}]
		If there exists an ordinal potential cost function $J\play{p}$ for a given $\Gamma_\text{LQ}$, for which
		\begin{equation} \label{eq:opdg_def}
			{\rm sgn}_\mathrm{v}\left(\frac{\partial H^{(p)}}{\partial \boldsymbol{u}^{(i)}}\right)
			= 
			{\rm sgn}_\mathrm{v}\left(\frac{\partial H^{(i)}}{\partial \boldsymbol{u}^{(i)}}\right)\; \; \forall i \in \mathcal{P},
		\end{equation}
		holds, then the game is an OPDG. In \eqref{eq:opdg_def}, ${\rm sgn}_\mathrm{v}$ of the vector $\sv{a} = [a_1, a_2, a_3]^T$ is defined such as
		\begin{equation*}
			{\rm sgn}_\mathrm{v}\left(\sv{a}\right) := \begin{bmatrix}
				\mathrm{sgn}(a_1) \\
				\mathrm{sgn}(a_2) \\
				\mathrm{sgn}(a_3)
			\end{bmatrix}
		\end{equation*}
		the element-wise sign function of the vector elements.
	\end{deff}
	 While the core idea of the subclass of OPDGs has been presented, still, a systematic identification method for OPDGs is still absent from the literature, which is a crucial step in the analysis and design of potential games. Thus, Problem \ref{prob:1} is defined as follows:
	\begin{prob_stat} \label{prob:1}
		Let an LQ differential game $\Gamma_\text{LQ}$ be given. How can be identified an OPDG for this associated LQ differential game, which is characterized by a quadratic potential function $J\play{p}$, as given in (\ref{eq:odpg_obj})?  
	\end{prob_stat}

	\section{Identification of OPDG}
	In this section, the main results of the paper are presented: Two novel identification methods to identify an OPDG for a given LQ differential game. 
	\subsection{Trajectory-free LMI Optimization}
	The identification method presented in this section is referred to as \textit{Trajectory-Free Optimization} (TFO).
	The TFO identifies an ordinal potential LQ differential game by solving an LMI optimization problem based on the idea from \cite{2015_SolutionsInverseLQR_priess}. This has the advantage that the ordinal potential cost function can be computed without determining the trajectories of the original game. This makes the optimization robust against measurement noise and disturbance.
	According to \cite{2021_PotentialDifferentialGames_varga}, the condition for an OPDG, cf.~(\ref{eq:opdg_def}), can be reformulated, leading to Lemma~1. 
	\begin{lem}[Sufficient Condition of an OPDG~\cite{2021_PotentialDifferentialGames_varga}] \label{thm-2} 
		If for a two-player linear-quadratic game,
		\begin{align} \label{eq:cond_ccat2021}
			\left( \vek{B}\Tplay{i}\vek{P}\play{p}\sv{x} \right)^\mathsf{T} \odot \left( \vek{B}\Tplay{i}\vek{P}\play{i}\sv{x}\right) \geq 0 
		\end{align}
		holds $\forall i \in \mathcal{P},$ and $\forall \sv{x}$, then it is an \textit{ordinal potential differential game} with the potential function given by (\ref{eq:odpg_obj}). The operation $\odot$ is the Schur product, defined as
		$$
		\begin{bmatrix}
			a_1 \\
			a_2 \\
			a_3
		\end{bmatrix}
	\odot
			\begin{bmatrix}
		b_1 \\
		b_2 \\
		b_3
	\end{bmatrix} = 		\begin{bmatrix}
	a_1 \cdot b_1\\
	a_2 \cdot b_2\\
	a_3 \cdot b_3
\end{bmatrix}
		$$
	\end{lem}
	\begin{proof}
		See \cite{2021_PotentialDifferentialGames_varga}.
	\end{proof}
	To verify whether the differential game is an OPDG, condition \eqref{eq:cond_ccat2021} has to be checked for all possible trajectories $\sv{x}$ of the original game. Consequently, using (\ref{eq:cond_ccat2021}) as the constraints of an optimization problem indicate a trajectory dependency of the identification. Therefore, in the following, the elimination of this trajectory dependency is discussed. First, the following notation is introduced:
	$$ 
	\vek{V}\play{p}:= \mathbf{B}^{(i)^{T}} \mathbf{P}^{(p)} \text{ and }\vek{V}\play{i}:= \mathbf{B}^{(i)^{T}} \mathbf{P}^{(i)},
	$$
	where $\vek{B}\play{i} \in \mathbb{R}^{p_i \times n}$, $\vek{P}\play{i}, \vek{P}\play{p} \in \mathbb{R}^{n \times n}$ consequently $\vek{V}\play{i}, \vek{V}\play{p} \in \mathbb{R}^{p_i \times n}$. Using this notation, (\ref{eq:cond_ccat2021}) can be rewritten as
	\begin{align}  \label{necessary_opdg_simplified}
		\left(\vek{V}\play{p} \sv{x} \right)^\mathsf{T} \odot \left(\vek{V}\play{i} \sv{x} \right)\geq \sv{0}.
	\end{align}
	In order to drop $\sv{x}$ in the constraint, we recall that (\ref{necessary_opdg_simplified}) must hold $\forall \sv{x}$, which means both terms in (\ref{necessary_opdg_simplified}) have the same sign regardless of the actual $\sv{x}$ leading to the new condition
	\begin{equation} \label{eq:paralell_optim_1}
		\sv{\omega}\play{i}\vek{V}\play{p}-\vek{V}\play{i} = \sv{0},
	\end{equation}
	where $\sv{\omega}\play{i} = \mathrm{diag}\left[\omega\play{i}_1, \omega\play{i}_2, ..., \omega\play{i}_n\right]$ is a arbitrary scaling matrix, in which $\omega\play{i}_i > 0$ holds. Note that \eqref{eq:paralell_optim_1} is a sufficient condition for \eqref{necessary_opdg_simplified} and not necessary, therefore, \eqref{eq:paralell_optim_1} is more restrictive than \eqref{necessary_opdg_simplified}.
	
	Figure~\ref{fig:method_1} represents the system state vector $\sv{x}$ in two different time instances $t_1$ and $t_2$. Furthermore, the matrices $\vek{V}\play{i}$ and $\vek{V}\play{p}$ are given cf.~\eqref{eq:paralell_optim_1}. In this three-dimensional example with scalar inputs ($\sv{x} \in \mathbb{R}^{3 \times 3}$ and $\vek{B}\play{i} \in \mathbb{R}^{1 \times 3}$), $\omega$ is a single scalar value and $\vek{V}\play{i}, \vek{V}\play{p}$ are vectors such as $\sv{v}\play{i}, \sv{v}\play{p}$, see Figure~\ref{fig:method_1}. In this example, $\sv{v}\play{i}, \sv{v}\play{p}$ have to show in the same direction in order to fulfill condition \eqref{eq:paralell_optim_1} for all time instances. In this example, intuitively, if $\sv{x}$ has a lower dimension than $\sv{v}\play{i}, \sv{v}\play{p}$ have, there are more possible combinations of $\sv{v}\play{p}$ and $\sv{v}\play{i}$, which fulfill condition \eqref{eq:paralell_optim_1}. For instance, if $\sv{x}$ is scalar and coincides with $\vek{e}^1$, then any arbitrary vectors in the plane $\vek{e}^2-\vek{e}^3$ fulfill condition \eqref{eq:paralell_optim_1}. This consideration holds for higher dimensions.

	To construct an LMI optimization for finding OPDG, the idea from \cite{2015_SolutionsInverseLQR_priess} is applied. This leads to a reformulation of the original problem statement (cf.~Problem \ref{prob:1}) such~as:
	\hspace*{1mm}
	\begin{prob_stat} \label{prob_modifed}
		Let a stabilizing feedback control $\sv{K}\play{p}$ and the system dynamics (\ref{eq:linear_system}) are given. Find (at least) one cost function \eqref{eq:odpg_obj} whose minimization with respect to the system dynamics leads the optimal control law $\sv{K}\play{p}$, such that \eqref{eq:opdg_def} holds.
	\end{prob_stat}
	\hspace*{1mm}
	\begin{collar}
		Problem \ref{prob_modifed} raises the uniqueness question: If parameters of (\ref{eq:odpg_obj}), $[\tilde{\vek{Q}}\play{p},\,\tilde{\vek{R}}\play{p}]$ solve Problem \ref{prob_modifed} then $[\gamma \cdot \tilde{\vek{Q}}\play{p},\,\gamma \cdot \tilde{\vek{R}}\play{p}]$ solve Problem 2 as well since they lead to the same optimal control law $\vek{K}\play{p}$, where $\gamma>0$ is a scalar. To overcome this issue and to ensure a unique solution further criteria must be provided.
	\end{collar}
	\begin{figure}[t!]
		\centering
		\includegraphics[width=0.83\linewidth]{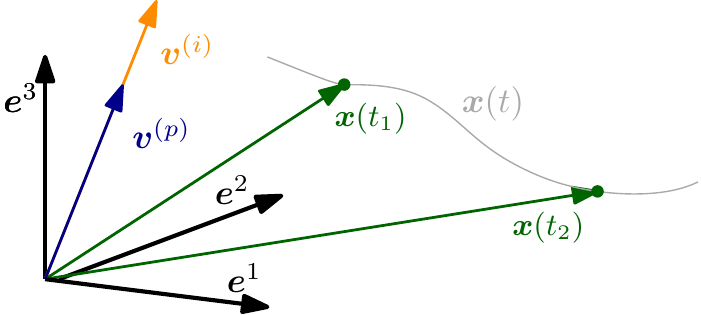}
		\caption{
			A schematic representation of the trajectory independency of the optimization in a three-dimensional space, where $\left(\sv{e}^1,\sv{e}^2,\sv{e}^3\right)$ are the basis vectors of the coordinate system. The vectors $\sv{v}\play{p}$ and $\sv{v}\play{i}$ show in the same direction and thus are linearly dependent. This means that the condition (\ref{necessary_opdg_simplified}) is automatically fulfilled at $t_1$ as well as at any other time $t_2$.}
		\label{fig:method_1}
	\end{figure}
	To achieve a unique solution for the associated inverse problem, we minimize the condition number $\alpha$ of the weighting matrices
	\begin{equation} \label{eq:cond_numb_def}
	\mathbf{I} \preceq
	\begin{bmatrix}
		\mathbf{Q}^{(p)} & \boldsymbol{0} \\
		\boldsymbol{0} & \mathbf{R}^{(p)}
	\end{bmatrix}
	\preceq \alpha \mathbf{I},
	\end{equation}
	thereby resolving the ambiguous scaling issue. The main difference to other inverse control problems (see \eg \cite{2015_SolutionsInverseLQR_priess} or \cite{2019_InverseDiscountedbasedLQR_el-hussieny}) is that (\ref{eq:cond_ccat2021}) or the reformulated condition \eqref{eq:paralell_optim_1} need to hold additionally. By summarizing, the following optimization problem is constructed:
	\newline
	\begin{subequations} \label{eq:optim1_lmi}
		\begin{align}
			\hat{\mathbf{P}}^{(p)},\hat{\mathbf{Q}}^{(p)}, \hat{\mathbf{R}}^{(p)}, \hat{\sv{\omega}}\play{i} &=\arg  \min _{\substack{\mathbf{P}^{(p)},\mathbf{Q}^{(p)}, \\  \mathbf{R}^{(p)}, \sv{\omega}\play{i}}}\alpha^{2} 
			\label{first_method_obj}\\
			{\rm s.t.} \quad \mathbf{A}^{T} \mathbf{P}^{(p)}+\mathbf{P}^{(p)} \mathbf{A}-\mathbf{P}^{(p)} \mathbf{B} \mathbf{K}\play{p}+\mathbf{Q}^{(p)}&=\mathbf{0} 
			\label{first_method_con1}\\
			\mathbf{B}^\mathsf{T} \mathbf{P}^{(p)}-\mathbf{R}^{(p)} \mathbf{K}\play{p} &= \mathbf{0} 
			\label{first_method_con2}\\
			\sv{\omega}\play{i} \mathbf{B}^{(i)^\mathsf{T}}\mathbf{P}^{(i)}-
			\mathbf{B}^{(i)^\mathsf{T}}\mathbf{P}^{(p)} &= \mathbf{0} \; \; \forall i\in \mathcal{P}
			\label{first_method_con3}\\
			\mathbf{I} \preceq
			\begin{bmatrix}
				\mathbf{Q}^{(p)} & \boldsymbol{0} \\
				\boldsymbol{0} & \mathbf{R}^{(p)}
			\end{bmatrix}
			&\preceq \alpha \mathbf{I} 
			\label{first_method_con4}\\
			\quad \mathbf{P}^{(p)} &\succeq \boldsymbol{0},
			\label{first_method_con5}
		\end{align}
	\end{subequations} 
	where $\vek{B} = [\vek{B}\play{1},...,\vek{B}\play{N}]$. The constraints (\ref{first_method_con1}), (\ref{first_method_con2}) and (\ref{first_method_con5}) are necessary that $\sv{K}\play{p}$ is optimal for the identified quadratic cost function. 
	The constraint (\ref{first_method_con4}) ensures the uniqueness of the solution, see \eqref{eq:cond_numb_def}.
	The constraint (\ref{first_method_con3}) restricts the identified cost function to an ordinal potential function. Thus, if \eqref{eq:optim1_lmi} is feasible, then Problem \ref{prob_modifed} is solved and the original game is an OPDG. Note since $\sv{\omega}\play{i}$ contains arbitrary values, it provides additional degrees of freedom for the optimization \eqref{eq:optim1_lmi}.
	\subsection{Feasibility Analysis}
	This section provides a feasibility analysis of \eqref{eq:optim1_lmi} for $\Gamma_{LQ}$ with two players. 
	Due to the constraints in \eqref{eq:optim1_lmi}, there are differential games, for which \eqref{eq:optim1_lmi} does not yield a feasible solution, since the constraints are violated: the solution is called \textit{non-feasible}. Thus, in the following, we discuss conditions for the feasibility of the proposed LMI optimization problem \eqref{eq:optim1_lmi}.
	
	Without loss of generality, it is assumed that the system has $n$ states and that the input matrices of player $1$ and player $2$, $\mathbf{B}^{(i)}, \; i=\{1,2\}$ have the dimensions $n \,\times\, p_1$ and $n \,\times\, p_2$, respectively. The TFO can provide feasible solutions if the following requirements are satisfied:
	\begin{lem}[Necessary Condition for the feasibility of the TFO for OPDGs] \label{lem:feas}
		The TFO (\ref{eq:optim1_lmi}) for two players can be feasible, only~if 
		\begin{itemize} \label{conditio:OPDG_LMI}
			\item[A)] the columns of the input matrices $\mathbf{B}^{(i)}, \; \forall i \in \mathcal{P}$ are linearly independent and 
			\item[B)] the system dimensions satisfy
			\begin{equation} \label{eq:lemma1_condition_opdg}
				\frac{1}{2}(1+n)-(p_1+p_2)>0,
			\end{equation}
			where $p_1$ and $p_2$ are the dimensions of the inputs vectors $\vek{B}\play{1} \in \mathbb{R}^{n \times p_1}$, $\vek{B}\play{2} \in \mathbb{R}^{n \times p_2}$ of player 1 and 2, respectively.
		\end{itemize}
	\end{lem}
	
	\begin{proof}
		To prove the conditions, constraint (\ref{first_method_con3}) is rewritten as
		$$
		\sv{\omega}\play{i} \vek{V}\play{i}- \mathbf{B}^{(i)^\mathsf{T}}\mathbf{P}^{(p)} = \mathbf{0}, \forall i \in \mathcal{P},
		$$ 
		which can be vectorized such that
		\begin{align}  \label{eq:pr_step2}
			\mathrm{vec}\left(\sv{\omega}\play{i} \vek{V}\play{i}\right) - \mathrm{vec}\left({\vek{B}\play{i}}^\mathsf{T} \vek{P}\play{p}\right) = \sv{0}, \forall i \in \mathcal{P},
		\end{align}
		and rearranged to
		\begin{equation} \label{eq:chap4_proof_feasibly}
			\underbrace{\left(\vek{E}_n   \otimes {\vek{B}\play{i}}^\mathsf{T} \right)}_{\tilde{\vek{A}}} \underbrace{\mathrm{vec} \left(\vek{P}\play{p} \right)}_{\tilde{\vek{x}}} = \underbrace{\mathrm{vec}\left(\sv{\omega}\play{i} \vek{V}\play{i}\right)}_{\tilde{\vek{b}}}, \forall i \in \mathcal{P},
		\end{equation}
		where $\mathrm{vec}(\cdot)$ represents the column vectorization of a matrix, $\otimes$ is the Kronecker product of two matrices and the matrix $\vek{E}_n$ is a identity matrix with an appropriate dimension. In \eqref{eq:chap4_proof_feasibly}, the classical form of a system of linear equations $
		\tilde{\vek{A}} \tilde{\vek{x}} = \tilde{\vek{b}}
		$ is given in the underbraces. 
		
		Condition A is necessary for the consistency of the system of linear equations \eqref{eq:chap4_proof_feasibly}, for which
		$$
		\mathrm{rank} \left(\vek{E}_{n}   \otimes {\vek{B}\play{i}}^\mathsf{T} \right) = 	\mathrm{rank} \left(\vek{E}_{n}   \otimes {\vek{B}\play{i}}^\mathsf{T} \, \Big| \,  \mathrm{vec}\left(\sv{\omega}\play{i} \vek{V}\play{i}\right) \right), \; \forall i \in \mathcal{P}
		$$
		must hold, since inconsistency of \eqref{first_method_con3} leads to an infeasible LMI.
		
		Condition B is necessary for the following reasons. 
		If \eqref{eq:chap4_proof_feasibly} yields a single solution for a given $\mathrm{vec}\left(\sv{\omega}\play{i} \vek{V}\play{i}\right)$, then \eqref{eq:chap4_proof_feasibly} completely determines $\mathrm{vec}(\vek{P}\play{p})$. Therefore, it is not possible to modify $\vek{P}\play{p}$ to satisfy (\ref{eq:optim1_lmi}b) and (\ref{eq:optim1_lmi}c), and as a result, (\ref{eq:optim1_lmi}) cannot be feasible. On the other hand, if \eqref{eq:chap4_proof_feasibly} has multiple solutions, the constraints of the optimization problem \eqref{eq:optim1_lmi} have additional degrees of freedom. This consideration requires a rank analysis of \eqref{eq:chap4_proof_feasibly} to show that Condition B holds, see e.g.~\cite[Chapter 5]{2014_LinearAlgebraMatrix_banerjee}. 
		Due to the fact that the columns of the input matrix $\mathbf{B}^{(i)}, \; \forall i \in \mathcal{P}$ are linearly independent, 
		\begin{equation} \label{eq:feas_rank_cond_fin}	
			\mathrm{rank} \left(\vek{E}_{n}   \otimes {\vek{B}\play{i}}^\mathsf{T} \right) < \mathrm{dim}\left( \mathrm{vec} \left(\vek{P}\play{p} \right) \right), \; \forall i \in \mathcal{P}
		\end{equation}
		must hold, where the dimension of a vector is denoted by $\mathrm{dim}$. Condition \eqref{eq:feas_rank_cond_fin} means that the number of rows of the coefficient matrix of \eqref{eq:chap4_proof_feasibly} need to be smaller than the number of columns, cf. \cite{2005_RowRankEquals_wardlaw, 2014_LinearAlgebraMatrix_banerjee}. For two players, \eqref{eq:chap4_proof_feasibly} leads to
		\begin{equation} \label{eq:2_player}			
			\tilde{\vek{A}} = 
			\begin{bmatrix}
				\vek{E}_n \otimes {\vek{B}\play{1}}^\mathsf{T} \\
				\vek{E}_n \otimes {\vek{B}\play{2}}^\mathsf{T}
			\end{bmatrix} \; \mathrm{and} \; \tilde{\vek{b}} = \begin{bmatrix}
				\sv{\omega}\play{1} \vek{V}\play{1} \\
				\sv{\omega}\play{2} \vek{V}\play{2}
			\end{bmatrix}.
		\end{equation} 
		The dimension of $\tilde{\vek{A}}$ is $n(p_1+p_2) \times n\cdot n,$ where $p_1$ and $p_2$ are the dimension of the input variables of player 1 and 2, respectively. If $\vek{P}\play{p}$ was not a symmetric matrix, the condition for a manifold of solutions would be $ n\cdot n > n(p_1+p_2)$. Due to the symmetric structure of $\vek{P}\play{p}$, the degrees of freedom of $\mathrm{vec}\left(\vek{P}\play{p}\right)$ are reduced to $\frac{1}{2}(1+n)n$, see \cite[Chapter 14]{2014_LinearAlgebraMatrix_banerjee}. Thus, condition \eqref{eq:feas_rank_cond_fin} is changed~to
		\begin{equation}
			\frac{1}{2}(1+n) > (p_1+p_2),
		\end{equation}
		which completes the proof.
		
	\end{proof}
	
	\begin{collar} \label{remark1}
		The constraint (\ref{first_method_con3}) is more restrictive as condition~(\ref{eq:opdg_def}) from the definition of an OPDG. Therefore, it is possible that the solution of (\ref{eq:optim1_lmi}) is infeasible while the original game admits being an OPDG. 
	\end{collar}
	
	\vspace*{2.5mm}
	\subsection{Weakly Trajectory-Dependent Optimization}
	If not all constraints of \eqref{eq:optim1_lmi} are fulfilled, the TFO yields an \textit{infeasible solution}. This can be the case, even if a potential function exists for the given differential game cf.~Remark~\ref{remark1}. To overcome this issue and to identify the potential game still, a second identification method is introduced, which uses only the relevant parts of the trajectory information for the constraints of the LMI. The optimization is called \textit{Weakly Trajectory-Dependent Optimization} (WTDO), since the constraints for an OPDG (\ref{eq:optim1_lmi}b~-~\ref{eq:optim1_lmi}f) are reformulated and the constraint on an exact solution to the Riccati equation \eqref{eq:coup_Ric} is softened. Instead of optimizing the condition number, the WTDO minimizes the remaining error of (\ref{eq:optim1_lmi}b). The constraints are reformulated such that the closest points around the zero crossing of~$\vek{B}\Tplay{i}\vek{P}\play{i}\sv{x}$, are computed, for which
	\begin{align}
		&\mathbf{B}^{(i)^{T}} \mathbf{P}^{(i)}\boldsymbol{x}^*_+ > \boldsymbol{0} \; \forall i \in \mathcal{P}, \\
		&\mathbf{B}^{(i)^{T}} \mathbf{P}^{(i)}\boldsymbol{x}^*_- < \boldsymbol{0} \; \forall i \in \mathcal{P}
	\end{align}
	hold, where $\boldsymbol{x}^*_+$ and $\boldsymbol{x}^*_-$ represent the limit points at the zero crossing of the trajectory $\sv{x}(t)$. This is a reasonable reformulation of constraints (\ref{eq:optim1_lmi}b), since the zero crossings --~the points, where the signs of \eqref{eq:opdg_def} change~-- are the points of interest to fulfill condition \eqref{eq:cond_ccat2021}. The constraint (\ref{first_method_con1}) from the TFO is softened and used for minimization\footnote{Note that reformulating hard constrained in a soft-constrained identification method is also used in literature, see \eg \cite{2020_InverseOpenLoopNoncooperative_molloya}.}. The WTDO is formulated as an LMI optimization,
	\begin{subequations}  \label{eq:optim2_lmi}
		\begin{align}
			\hat{\mathbf{P}}^{(p)},\hat{\mathbf{Q}}^{(p)},& \hat{\mathbf{R}}^{(p)}=\arg  \min _{\mathbf{P}^{(p)},\mathbf{Q}^{(p)}, \mathbf{R}^{(p)}}\eta^{2} \\
			&{\rm s.t.} \; \text{(\ref{first_method_con2}),} \text{ and (\ref{first_method_con5})}\\
			&\mathbf{I} \preceq
			\begin{bmatrix}
				\mathbf{Q}^{(p)} & \boldsymbol{0} \\
				\boldsymbol{0} & \mathbf{R}^{(p)}
			\end{bmatrix}\\
			&\mathbf{B}^{(i)^{T}}
			\mathbf{P}^{(p)}\boldsymbol{x}^*_+ > \boldsymbol{0} \; \forall i \in \mathcal{P},
			\label{second_method_con2}\\
			&\mathbf{B}^{(i)^{T}} \mathbf{P}^{(p)}\boldsymbol{x}^*_- < \boldsymbol{0} \; \forall i \in \mathcal{P},
			\label{second_method_con3}
		\end{align}
	\end{subequations} 
	where
	\begin{equation} \label{eq:trace_WTO}
		\eta = \text{tr}\left(\mathbf{A}^{T} \mathbf{P}^{(p)}+\mathbf{P}^{(p)} \mathbf{A}-\mathbf{P}^{(p)} \mathbf{B} \mathbf{K}\play{p}+\mathbf{Q}^{(p)}\right).
	\end{equation}
	Thus, the constraint (\ref{first_method_con3}) is changed to the optimization objective of the WTDO. Furthermore, (\ref{second_method_con2}) and (\ref{second_method_con3}) ensure the condition of the OPDGs. Since the WTDO is also formulated as an LMI, therefore an efficient calculation is guaranteed. Note that the usage of the trace \eqref{eq:trace_WTO} ensures the soft-constrained NE of the original game.
	
	\section{Applications}
	This section provides an academic and an engineering example to demonstrate the applicability of the proposed identification methods. Furthermore, the results are compared to the state-of-the-art solution.
	\subsection{Input Dependent Optimization}
	To establish a baseline for the analysis, we utilize the identification method proposed in \cite{2021_PotentialDifferentialGames_varga} and compare it with the two novel approaches. This state-of-the-art identification is referred to as \textit{Input-Dependent Optimization} (IDO).
	The IDO approach involves measuring the error between the inputs from the potential function and those from the original game ($\sv{x}^*$), which corresponds to its NE
	\begin{equation} \label{eq:error_optim}
		\sv{e}_{u} =  \sv{u}\play{p}(t,\sv{x},\vek{Q}\play{p},\vek{R}\play{p}) - \sv{u}(t,\sv{x}^*).
	\end{equation}
	To identify the parameters of the OPDG, this error is minimized, which is happens through the following optimization:
	\begin{subequations} \label{eq:optim_find_pot_games}
		\begin{align} 
			&\; \; \; \hat{\vek{Q}}\play{p}, \hat{\vek{R}}\play{p} = \text{arg } \underset{\vek{Q}\play{p},\vek{R}\play{p}}{\text{min }}  \left| 	\sv{e}_{u} \right|^2  \\ 
			\text{s.t. }&\vek{A}^\mathsf{T} \vek{P}\play{p} +  \vek{P}\play{p} \vek{A} + \vek{Q}\play{p} - \vek{P}\play{p}\vek{S}\play{p}\vek{P}\play{p} = \sv{0} \\
			&\left( \vek{B}\Tplay{i}\vek{P}\play{p}\sv{x} \right)^\mathsf{T} \cdot \left( \vek{B}\Tplay{i}\vek{P}\play{i}\sv{x}\right) \geq 0 \; \forall i \in \mathcal{P},
		\end{align}
	\end{subequations}
	where $\vek{S}\play{p} = \vek{B}\play{p}\vek{R}\Invplay{p}\vek{B}\Tplay{p}$.
	The minimization of (\ref{eq:optim_find_pot_games}a) ensures that the necessary condition OPDG cf.~\cite{2021_PotentialDifferentialGames_varga}. The constraint (\ref{eq:optim_find_pot_games}b) is necessary for the optimum of $J\play{p}$ meaning that $\sv{u}\play{p}$ is the result of the optimal control problem. The constraint (\ref{eq:optim_find_pot_games}c) guarantees the sufficient condition of Lemma~\ref{eq:cond_ccat2021}. 
	The optimizer is an interior-point optimizer algorithm provided by MATLAB \cite{MATLAB:R2019b_u5}.
	
	In the next section, two examples are presented. The first one is a general example: Neither the players' cost functions nor the system dynamics are special in contrast the the state-of-the-art examples. The second one is a specific engineering application, in which a human-machine interaction is modeled as an OPDG. 
	
	\subsection{General Example}
	\subsubsection{The LQ differential game}
	In the first example, a linear time-invariant model is used, in which the system and input matrices are the following:
	\begin{equation*}
		\mathbf{A} \! =\! \begin{bmatrix}
			-1.20 & 0.00 & 0.00 & 0.00 & 0.00 & 1.75 \\
			0.00 & 2.10 & 0.00 & 0.00 & 0.00 & 0.00 \\
			-1.00 & 0.00 & 2.95 & 0.00 & 0.00 & 0.00 \\
			0.00 & 0.00 & 0.00 & 2.05 & 0.00 & 1.50 \\
			2.00 & 0.00 & 0.00 & 0.00 & 1.00 & -4.15 \\
			0.00 & 0.00 & 0.00 & 0.00 & 0.00 & 1.85
		\end{bmatrix}\!, \;
		\mathbf{B}\play{1} =\! 
		\begin{bmatrix}
			1.0 & 1.0 \\
			3.0 & 0.0 \\
			2.1 & 4.0 \\
			0.0 & 0.0 \\
			0.1 & 2.0 \\
			1.0 & 0.9
		\end{bmatrix}
		\text{and} \,
		\mathbf{B}\play{2} = \!
		\begin{bmatrix}
			1.3 & 1.0 \\
			1.0 & -1.1 \\
			0.0 & 0.0 \\
			2.0 & -1.0 \\
			0.0 & -2.0 \\
			4.0 & 2.1
		\end{bmatrix}\hspace*{-1mm}.
	\end{equation*}
	Note that the system and input matrices
	Both players have a quadratic cost function, cf. (\ref{eq:each_player_quad_cost_function}). The penalty factors of the first player are
	\begin{align*}
		\mathbf{Q}^{(1)} &= \rm{diag}([10,4,2,3,4,4]),\nonumber \\
		\mathbf{R}^{(11)} &= \rm{diag}([1.5,1.0]), \; \nonumber 
		\mathbf{R}^{(12)} = \rm{diag}([0,0]) \nonumber 
	\end{align*}
	and the factors of the second player are
	\begin{align*}
		\mathbf{Q}^{(2)} &= \rm{diag}([8,1,5,1,3,2]), \nonumber \\ 
		\mathbf{R}^{(21)} &= \rm{diag}([0.1,0]), \; \nonumber
		\mathbf{R}^{(22)} = \rm{diag}([1,1]). \nonumber
	\end{align*}
	The initial values of the simulation are	$x_0 = [-0.5,\,-1.9,\,0.8, \, -0.6, \, 2.9, \,-0.1]^\mathsf{T}.$   
	The control laws of the players are computed by the coupled optimization of~(\ref{eq:each_player_quad_cost_function}), with $i=1,2$, in which the coupled algebraic Riccati equation is solved iteratively. The obtained feedback gains of the players are
	\begin{align*}
		\vek{K}_1 &=   \begin{bmatrix}
			-0.90 & 2.26 & 1.03 & -0.55 & -0.80 & 0.40 \\
			-2.94 & -1.04 & 3.91 & 1.43 & -0.81 & 0.89
		\end{bmatrix}, \\
		\vek{K}_2 &=   \begin{bmatrix}
			-0.92 & -0.25 & 0.69 & 2.71 & -1.44 & 2.04 \\
			-0.45 & -0.55 & 0.65 & -0.78 & -1.53 & 1.31
		\end{bmatrix}.
	\end{align*}
	
	\subsubsection{Results}
	For the evaluation and comparison of the identification methods, the error in the state trajectory is defined as 
	\begin{align} \label{eq:results_error}
		e^{x} \; &=\max \left\{e^{x_1}, {e^{x_2}} \ldots, e^{x_n}\right\}, \\ \nonumber
		e^{x_i} &=\left\|\frac{{\sv{x}}_{i}\play{p}}{\left\|{\sv{x}}_{i}\play{p}\right\|_{\max }}-\frac{{\sv{x}}_{i}^*}{\left\|{\sv{x}}_{i}\play{p}\right\|_{\max }}\right\|_{\max }, 
		\;\forall i \in\{1, \ldots, n\},
	\end{align}
	where $\sv{x}_{i}^*$ and $\sv{x}_{i}\play{p}$ are the trajectories generated by the original game (OG) and by the OPDG, respectively. Furthermore, the computation time necessary for the identification is used for the evaluation. The OG means that the inputs of the LQ differential game are computed by 
	$$
	\sv{u}\play{i}(t) = -\sv{K}\play{i}\sv{x}(t),
	$$
	where $\sv{K}\play{i}$ is obtained from \eqref{eq:KfromRicc_res}.
	
	The numerical results of $\vek{Q}\play{p}$ and $\vek{R}\play{p}$ using (\ref{eq:optim1_lmi}) are 
	\begin{equation*}
		\vek{Q}\play{p} = \begin{bmatrix}
			16.75 & -1.26 & -2.62 & -3.88 & 0.73 & 4.11 \\
			-1.26 & 5.16 & 1.17 & 0.70 & 0.56 & 0.85 \\
			-2.62 & 1.17 & 6.10 & 0.43 & -0.26 & -0.15 \\
			-3.88 & 0.70 & 0.43 & 6.72 & 0.56 & 1.91 \\
			0.73 & 0.56 & -0.26 & 0.56 & 11.21 & -1.02 \\
			4.11 & 0.85 & -0.15 & 1.91 & -1.02 & 2.87
		\end{bmatrix}
	\end{equation*}
	\begin{equation*}
		\vek{R}\play{p} =   
		\begin{bmatrix}
			2.12 & 0.39 & 0.32 & -0.08 \\
			0.39 & 2.06 & -0.07 & -0.10 \\
			0.32 & -0.07 & 3.22 & -0.87 \\
			-0.08 & -0.10 & -0.87 & 6.84
		\end{bmatrix}
	\end{equation*}
	Figure~\ref{fig:sys_states13} and Figure~\ref{fig:sys_states46} show the resulting trajectories of the original game, with two players (solid lines), and of the substituting potential game (dashed lines), which are the result of \eqref{eq:optim2_lmi}. It can be seen that the trajectories from the controller designed via the potential cost function deviate insignificantly.	
	\begin{figure}[b!]
		\centering
		\begin{subfigure}[t]{0.95\textwidth}
			\includegraphics{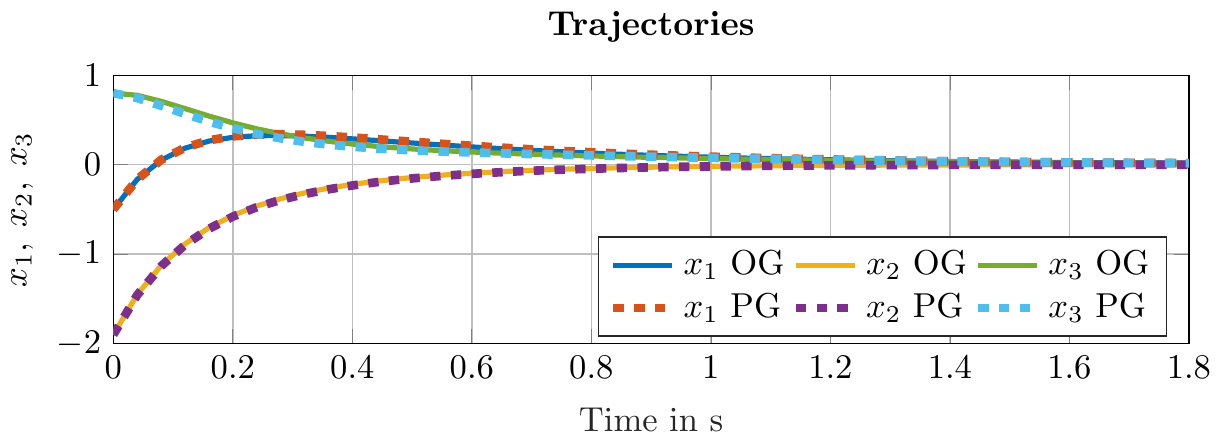}
			\caption{The 1.-3. system state trajectories}
			\label{fig:sys_states13}
		\end{subfigure}
		\hfill
		\begin{subfigure}[t]{0.95\textwidth}
			\includegraphics{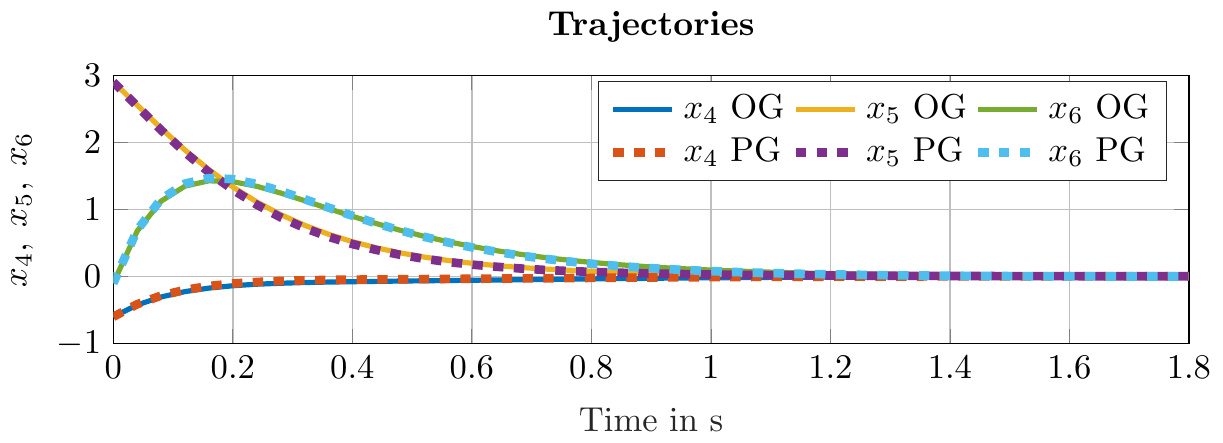}
			\caption{The 4.-6. system state trajectories}
			\label{fig:sys_states46}
		\end{subfigure}
		
		\caption{Comparison of the system state trajectories of the original game (solid lines) and the trajectories of the potential game (dashed lines)}
	\end{figure}

	\begin{figure}[t!]
		\centering 
		\includegraphics[width=0.99\linewidth]{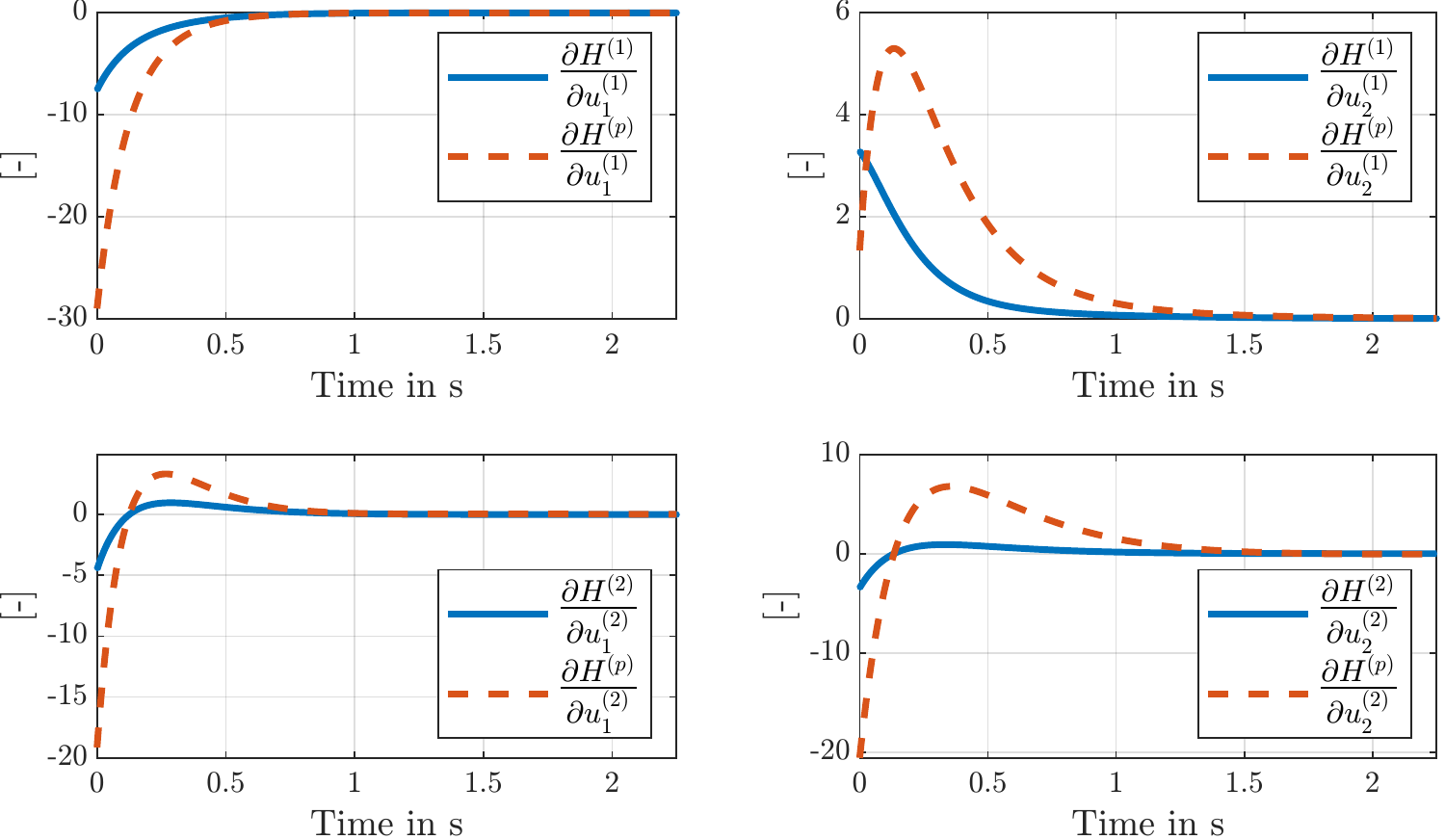}
		\caption{The derivatives of the Hamiltonian functions of the two players in accordance to \eqref{eq:opdg_def}}
		\label{fig:ham1}
	\end{figure}
	Figure~\ref{fig:ham1} shows the value of the derivatives of the Hamiltonian 
	$$\frac{\partial H^{(p)}}{\partial u^{(i)}_j} \; \mathrm{and} \; \frac{\partial H^{(i)}}{\partial u^{(i)}_j}.$$ 
	It can be seen that all the zero-crossing points for the OG and the identified OPDG occur at the same time. This verifies the ordinal potential structure of $\Gamma_{\text{LQ}}$ meaning that (\ref{eq:opdg_def}) holds $\forall t$.
	
	Table \ref{table:three2compare} compares the results of TFO and WTDO with IDO: It can be seen that the fastest and most accurate results are generated by the novel TFO. The trajectory errors $e^x$ are not different from IDO and WTDO. On the other hand, WTDO requires less computation time compared to state-of-the-art IDO. The two novel identification methods outperform the state-of-the-art solution. 
	
	\begin{table}[!h] 
		\centering
		\normalsize 
		\caption{The trajectory error and the computational times of the three identification methods}
		\begin{tabular}{c|ccc}
			
			& \; \; TFO \; \; & \; \; WTDO \; \; &\; \; IDO \; \;\\
			\hline
			\hline
			$e^x$ [-]& 0.019 &0.077 & 0.076\\
			\hline
			$t_\mathrm{comp}$ in s& 0.15 & 28.15 & 212.1
		\end{tabular}
		\label{table:three2compare}
	\end{table}

	\subsection{Engineering Example}
	The second simulation example is an engineering application with practical relevance, which presents the longitudinal model of a vehicle manipulator \cite{2019_ModelPredictiveControl_varga}. Such systems are used for road maintenance works, in which a human operator and the automation control the system.
	This human-machine interaction can be formulated as a differential game $\Gamma_{\text{LQ}}$ with a linear system and two players enabling a systematic controller design, which enables the interaction between humans and machines.
	
	\subsubsection{The two-player differential game}
	The model of this engineering example for the longitudinal motion of a vehicle manipulator has the following system and input matrices
	\begin{equation*}
		\mathbf{A}=  \begin{bmatrix}
			0.0 & 1.0 & 0.0 \\
			0.0 & 0.0 & 0.0 \\
			0.0 & 1.0 & 0.0
		\end{bmatrix}\!, \, \mathbf{B}^{(\mathrm{h})} = 
		\begin{bmatrix}
			0.00 \\
			0.00 \\
			0.14
		\end{bmatrix} \, \mathrm{and} \,
		\mathbf{B}^{(\mathrm{a})} = 
		\begin{bmatrix}
			0.0 \\
			1.0 \\
			0.0
		\end{bmatrix}.
	\end{equation*}
	The initial states are ${\sv{x}_0 = -1.2, \; -0.95, \; 0.5]}$. The penalty factor $\mathbf{Q}^{(i)}$ and $\mathbf{R}^{(i)}$ in the cost function (\ref{eq:each_player_quad_cost_function}) of each player are
	\begin{align*}
		\mathbf{Q}^{(\mathrm{h})} &= \rm{diag}([1,1,5]),\nonumber \\
		\mathbf{R}^{(\mathrm{h})} &= 1, \; \nonumber 
		\mathbf{R}^{(\mathrm{ha})} = 0.25 \\
		\mathbf{Q}^{(\mathrm{a})} &= \rm{diag}([0.344, \;    0.076, \;    1.409]), \nonumber \\ 
		\mathbf{R}^{(\mathrm{ah})} &= 0.19, \; \nonumber
		\mathbf{R}^{(\mathrm{a})} = 1. \nonumber
	\end{align*}
	The feedback control law of the original game is calculated \eqref{eq:each_player_quad_cost_function}, leading to the feedback gains of the human and automation
	\begin{equation}
		K\play{\mathrm{h}} =   \begin{bmatrix}
			-0.78 & 0.26 & 1.42
		\end{bmatrix}, \; \;   K\play{\mathrm{a}} =   \begin{bmatrix}
			0.42 & 1.59 & 0.83
		\end{bmatrix}.
	\end{equation}
	Since, \eqref{eq:lemma1_condition_opdg} does not hold for this model, the TFO cannot be applied. Thus, WTDO and IDO are used to compute the potential function of the game. 
	
	In order to analyze the robustness of the WTDO, white Gaussian noise 
	\begin{equation*}
		\tilde{\sv{x}}(t) = \sv{x}^*(t) + \sv{\sigma}(t)
	\end{equation*}
	is added to the state signal. 
	Enabling analysis in scenarios resembling real-world setups, this modification involves the inclusion of signal noises in the system states 
	Note that this procedure does not aim to provide a stochastic game analysis, since that would require different mathematical tools and methods. 
	The results obtained from WTDO, identifying noise states, are compared with the IDO results
	\subsubsection{Results}
	The identification with (\ref{eq:optim2_lmi}) lead to the following matrices of the potential function
	\begin{align}
		\vek{Q}\play{p} =   \begin{bmatrix}
			0.82 & 0.24 & -0.48 \\
			0.24 & 0.59 & -1.01 \\
			-0.48 & -1.01 & 2.15
		\end{bmatrix} \; \mathrm{and} \; \vek{R}\play{p} =   \begin{bmatrix}
			1.00 & -0.05 \\
			-0.05 & 1.60
		\end{bmatrix}
	\end{align}	
	To compare the performance and the robustness of the WTDO with IDO, their error indices \eqref{eq:results_error} are computed at different noise levels. The results are shown in Table~\ref{table:res_traj_noise}. Upon closer inspection of the results, it becomes apparent that the novel WTDO demonstrates superior performance compared to the state-of-the-art IDO at all signal noise levels. Figure~\ref{fig:sys_states2Ex} shows the system trajectories with $10\,$dB SNR. It can be seen that the WTDO still provides a reliable solution, even at the noise level of $10\,$dB. 
	
	\begin{table}[!t] 
		\normalsize 
		\caption{Results with different white Gaussian noise levels}
		\centering
		\begin{tabular}{c|ccccc}		
			SNR in dB&$10$ & $20$ & $30$ & $40$ & $\infty$\\
			\hline
			\hline
			$e^x_\mathrm{WTDO} \, [-]$ & 0.314& 0.107& 0.029& 0.027&0.002\\
			\hline
			$e^x_\mathrm{IDO} \, [-]$& 0.603& 0.265& 0.104& 0.047& 0.026
		\end{tabular}
		\label{table:res_traj_noise}
	\end{table}
	
	\begin{figure}[t!]
		\centering
		\includegraphics{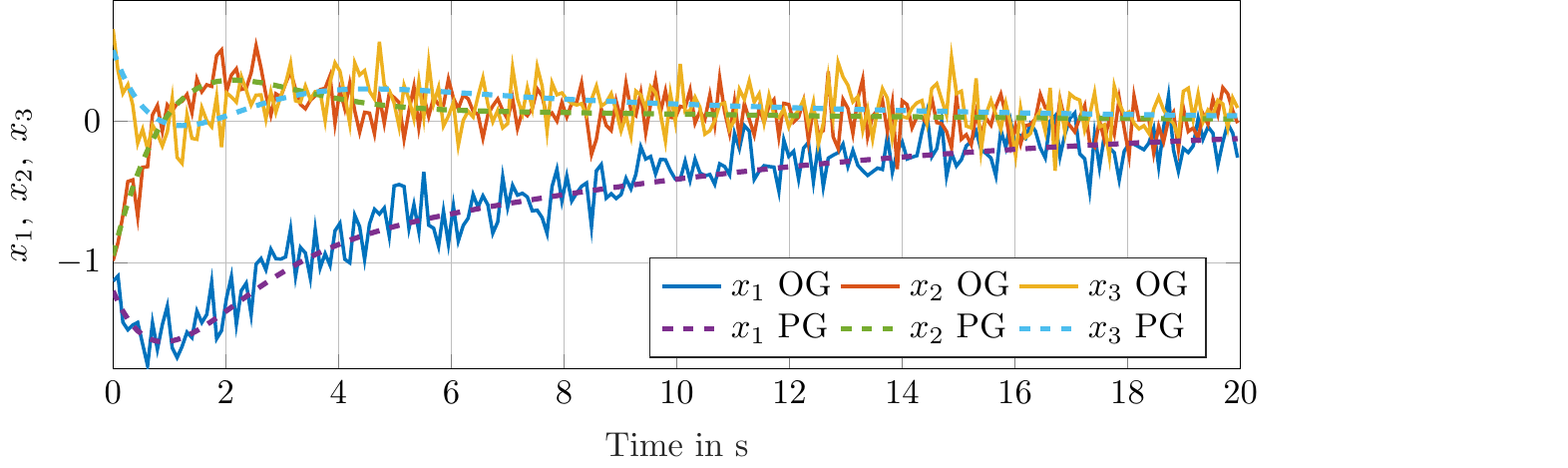}
		\caption{Comparison of the system state trajectories of the original game (solid lines) and the trajectories of the potential game (dashed lines)}
		\label{fig:sys_states2Ex}
	\end{figure}

	\subsection{Discussion}
	As the first simulation example showed, the TFO outperforms both WTDO and the state-of-the-art IDO. Both novel methods can provide an OPDG for the given differential game faster and more accurately compared to IDO. Furthermore, the second example explores the robustness of WTDO under varying noise levels. The findings demonstrate its potential for practical applications, such as modeling human-machine interaction as an OPDG, thereby indicating its suitability for real-world scenarios.
	
	However, a theoretical analysis of TFO regarding its computational complexity is not provided in this paper. Lemma \ref{lem:feas} provides only the necessary condition for the existence of an OPDG. Furthermore, the limitations of WTDO are not investigated in this work. Thus, these open research questions need to be addressed in further research.

	\section{Conclusion and Outlook}
	
	This paper has presented two systematic identification methods for finding an ordinal potential differential game corresponding to a given LQ differential game, addressing a gap in the existing literature. Both identification methods utilize linear matrix inequality optimization techniques. The first method -- referred to as trajectory-free optimization -- leverages only the cost functions of the original differential game to identify the ordinal potential game. In cases where trajectory-free optimization is infeasible, the second method, referred to as weakly trajectory-dependent optimization, is proposed as an alternative. Simulation results demonstrate that both identification methods effectively reconstruct the trajectories of the original game while satisfying the conditions of an ordinal potential differential game. Moreover, they exhibit superior speed, accuracy, and robustness compared to a previously proposed identification method from the literature.
	
	In future work, we plan to employ the proposed algorithms for designing cooperative learning controllers, see \eg \cite{2023_LimitedInformationShared_varga}. Additionally, we aim to validate the effectiveness of the proposed identification methods through measurements obtained from human-machine interactions.

	\section*{Declarations}
	\subsection*{Ethical Approval}
	Not applicable
	
	\subsection*{Competing Interests}
	The authors have no relevant financial or non-financial interests to disclose.

	\subsection*{Author Contributions}
	All authors contributed to the concept of the ordinal potential differential games. The implementation was carried out by Balint Varga and Da Huang. The first draft of the manuscript was written by Balint Varga. Da Huang and S\"oren Hohmann commented on previous versions of the manuscript. All authors read and approved the final manuscript.
	
	\subsection*{Founding}
	This work was supported by the Federal Ministry for Economic Affairs and Climate Action, in the New Vehicle and System Technologies research initiative with Project number 19A21008D.

	\subsection*{Availability of data and materials}
	The simulation results are available to readers on request to the corresponding author.


	

\end{document}